\documentclass[11pt,a4paper]{amsart}
\usepackage{amssymb}
\usepackage{latexsym}
\usepackage{exscale}
\usepackage{amsfonts}
\usepackage{graphicx}
\usepackage{mathrsfs}
\usepackage{amsmath,amscd,amsthm}
\usepackage{bbm,pifont}
\usepackage{enumerate}
\usepackage{color}
\usepackage{amssymb}
\usepackage{latexsym}
\usepackage{exscale}

\allowdisplaybreaks

\DeclareMathAlphabet\mathcalbf{OMS}{cmsy}{b}{n}
\DeclareMathAlphabet\EuScript{U}{eus}{m}{n}
\DeclareMathAlphabet\EuScriptBold{U}{eus}{b}{n}

\numberwithin{equation}{section}
\newtheorem{theorem}{Theorem}[section]
\newtheorem{lemma}[theorem]{Lemma}
\newtheorem{corollary}[theorem]{Corollary}

\newtheorem{remark}[theorem]{Remark}

\def\C{\mathbb C}

\begin{document}
\allowdisplaybreaks

\title[]
{An explicit formula of Cauchy--Szeg\"{o} kernel for quaternionic Siegel upper half space
and applications}
\author{Der-Chen Chang, Xuan Thinh Duong, Ji Li, Wei Wang and Qingyan Wu}

\address{Der-Chen Chang,
Department of Mathematics and Department of Computer Science,
Georgetown University, Washington D.C. 20057, USA
and
Graduate Institute of Business Administration, College of Management,
Fu Jen Catholic University, New Taipei City 242, Taiwan, ROC.}
\email{chang@georgetown.edu}

\address{Xuan Thinh Duong, Department of Mathematics, Macquarie University, NSW, 2109, Australia}
\email{xuan.duong@mq.edu.au}

\address{Ji Li, Department of Mathematics, Macquarie University, NSW, 2109, Australia}
\email{ji.li@mq.edu.au}

\address{Wei Wang, Department of Mathematics, Zhejiang University, Zhejiang 310027, China}
\email{wwang@zju.edu.cn}

\address{Qingyan Wu, Department of Mathematics\\
         Linyi University\\
         Shandong, 276005, China
         }
\email{qingyanwu@gmail.com}

\date{\today}

\begin{abstract}
In this paper we obtain an explicit formula of Cauchy--Szeg\"{o} kernel for quaternionic Siegel upper half space, and then based on this, we prove that the Cauchy--Szeg\"{o} projection on quaternionic Heisenberg group is a Calder\'on--Zygmund operator via verifying the size and regularity conditions for the kernel. Next, we also obtain a suitable version of pointwise lower bound
for the  kernel, which further implies the characterisations of the boundedness and compactness of commutator of the Cauchy--Szeg\"{o} operator via the BMO and VMO spaces on  quaternionic Heisenberg group, respectively.
\end{abstract}

\maketitle


\section{Introduction and statement of main results}
\setcounter{equation}{0}

\subsection{Background}

The quaternion numbers $x=x_1+x_2{\bf i}+x_3{\bf j}+x_4{\bf k}$ with ${\bf i}^2 = {\bf j}^2 = {\bf k}^2 = {\bf ijk} = -1$ were introduced
by Hamilton as an extension of complex numbers in the $19^{th}$ century and applied to mechanics in three-dimensional space. A feature of quaternions is that multiplication of two quaternions is noncommutative.   Let $\mathbb H$ be the system of quaternion numbers and let $\operatorname{Re}x$ and $\operatorname{Im}x$ denote the real part and imaginary part of $x$ respectively. Then $\operatorname{Re}x=x_1$ and $\operatorname{Im}x=x_2{\bf i}+x_3{\bf j}+x_4{\bf k}$. The $n$-dimensional quaternion space $\mathbb H^n$ is the collection of $n$-tuples $(q_1, \cdots, q_n)$.
An $\mathbb H$-valued function $f: \Omega\to\mathbb H$ is called regular if
$\bar\partial f =0,$ where
$$\bar\partial={\partial\over\partial x_1}+\sum_{m=1}^3{\bf i}_m {\partial\over\partial x_{m+1}} \ \ {\rm or}\ \
\bar\partial={\partial\over\partial x_1}+\sum_{m=1}^3 {\partial\over\partial x_{m+1}}{\bf i}_m,\quad
{\bf i}_1={\bf i}, ~{\bf i}_2={\bf j}, ~{\bf i}_3={\bf k}.$$
The theory of regular functions of several quaternionic variables began in 1980s. So far there are several fundamental results established for the quaternionic counterpart of the theory of several complex variables, e.g., Hartogs phenomenon, $k$-Cauchy--Fueter complexes, quaternionic pluripotential theory etc  (see for example \cite{alesker1,CSSS,wan-wang,wang19} and the references therein).
Because of the non-commutativity of the quaternion $\mathbb H$, many new phenomena emerge, e.g. the product of two regular functions is not regular in general. Hence, proofs of almost all results are completely different from the standard setting of complex variables.

Parallel to the setting of several complex variables,
people are also interested in the Hardy space of regular functions
over a bounded domain in $\mathbb H^n$, in particular, over the unit ball. By quaternionic Cayley transformation, it is equivalent to considering the Hardy space over the Siegel upper half domain
%
%
%
%
%
$$\mathcal U_n:=\left\{q=(q_1,\cdots, q_n)=(q_1,q')\in\mathbb H^n\mid \operatorname{Re} q_1>|q'|^2\right\},$$
where $q'=(q_2,\cdots,q_n)\in\mathbb H^{n-1}$, whose boundary
$\partial \mathcal U_n:=\{(q_1,q')\in\mathbb H^n\mid \operatorname{Re}q_1=|q'|^2\}$
 is a quadratic hypersurface, which can be identified with the quaternionic Heisenberg group
 $\mathscr H^{n-1}$.
 For the notation and details of definition we refer to Section 2.

 We point out that the quaternionic Heisenberg group  $\mathscr H^{n-1}$
 plays the fundamental role in quaternionic analysis and geometry \cite{Chr,Ivanov2,Wa11,shi-wang16}. Its analytic and geometric
behaviours are different from the usual Heisenberg group in many aspects, e.g., there does not exist nontrivial quasiconformal mapping between the quaternionic Heisenberg group \cite{Pansu} while
quasiconformal mappings between Heisenberg groups are abundant \cite{Koranyi1,Koranyi2}.

 The identification $\partial \mathcal U_n$ with the quaternionic Heisenberg group $\mathscr H^{n-1}$  via the projection \eqref{p} helps us to determine the kernel of the Cauchy--Szeg\"o projection from
 $L^2(\partial \mathcal U_n)$ to $H^2(\partial \mathcal U_n)$ \cite{CM08}.
 This was just obtained  recently  by the first, fourth authors and Markina
 \cite{CMW}. To be more explicit, we recall the result as follows.

 \medskip\noindent
{\bf Theorem A (\cite{CMW}).} {\it
The Cauchy--Szeg\"o kernel is given by
\begin{align}\label{cauchy-szego}
S(q,p)=s\Big(q_1+\overline p_1-2\sum_{k=2}^n\overline p_kq_k\Big),
\end{align}
for $p=(p_1,p')=(p_1,\cdots, p_n)\in\mathcal U_n$, $q=(q_1,q')=(q_1,\cdots, q_n)\in\mathcal U_n$, where
\begin{align}\label{s}
s(\sigma)=c_{n-1}{\partial^{2(n-1)}\over \partial x_1^{2(n-1)}}{{\overline \sigma}\over |\sigma|^4},\quad
\sigma=x_1+x_2{\bf i}+x_3{\bf j}+x_4{\bf k}\in\mathbb H,
\end{align}
with the real constant $c_{n-1}$ depending only on $n$.
%
%
The Cauchy--Szeg\"o kernel satisfies the reproducing property in the following sense
$F(q)=\int_{\partial \mathcal U_n} S(q,\xi)F^b(\xi)d\beta(\xi),\ q\in\mathcal U_n,
$
whenever $F\in \mathcal H^2(\mathcal U_n)$ and $F^b$ its boundary value on $\partial \mathcal U_n$.
}
\medskip

Recall that the standard Cauchy--Szeg\"o projection on the usual Heisenberg group has an explicit formula, which, together with the appropriate quasi-metric and the volume function, implies that the Cauchy--Szeg\"o kernel on the Heisenberg group is a Calder\'on--Zygmund kernel (we refer to \cite[Section 2, Chapter XII]{St} for background and  details). Hence, Stein  (\cite[Theorem 2, Chapter XII]{St}) summarised and stated the main conclusion as:

``The Cauchy--Szeg\"o projection on the Heisenberg group has an extension to a bounded operator on $L^p$ for all $1<p<\infty$.''

Moreover, it is also known that from the explicit formula of the Cauchy--Szeg\"o kernel, especially the pointwise lower bound, one can obtain the characterisation of the boundedness and compactness of the commutator of the Cauchy--Szeg\"o projection via the BMO and VMO spaces, respectively. See the recent results in \cite[Section 7.2]{DGKLWY} and \cite{CDLW}, respectively.
\medskip

It is natural to ask, whether the Cauchy--Szeg\"o kernel on the quaternionic Heisenberg group determined in Theorem A above has an explicit formula and a Calder\'on--Zygmund theory parallel to \cite[Section 2, Chapter XII]{St}?

\smallskip

\subsection{Statement of main results}

The aim of this paper is to address this question, providing
an explicit formula of the Cauchy--Szeg\"o kernel for the quaternion Siegel upper half-space. Then by using this formula we verify that it is a standard Calder\'on--Zygmund kernel with respect to the  quasi-metric $\rho$ (defined in Section 2, which is a standard definition in general stratified Lie group, see for example \cite[Chapter 1]{FS}), that is, it satisfies the standard size, and smoothness conditions in terms of $\rho$. Moreover, we also provide a suitable version of pointwise kernel lower bound from the explicit formula of the kernel, which was motivated by recent works on commutator estimates such as \cite{Hy, DGKLWY, DLLW}.

 As a direct application, we establish the characterisation of the BMO space via the commutator $[b,\mathcal C]$, which is parallel to the classical setting in \cite{CRW} for the Riesz transforms on $\mathbb R^n$, and some closely related recent results in \cite{HLW} for two weight commutators, \cite{LNWW} for the Cauchy integral along Lipschitz curves, \cite{DLLW} for the Riesz transform on stratified Lie groups, \cite{LOR,LOR2,GLW} for singular integrals with rough kernels, \cite{Hy} for the $L^p$ to $L^q$ boundedness of commutators with applications to the Jacobian operator, \cite{DLLWW} for the Cauchy integrals on certain pseudoconvex domains in $\mathbb C^n$ with minimal regularity condition on the boundary. See also the general frame on space of homogeneous type in \cite{DGKLWY}, which restructured the proof of \cite{DLWY,DHLWY,LW} and the references therein. Note that the weakest sufficient condition known so far on the kernel lower bound of Calder\'on--Zygmund operators which gives the characterisation of the BMO via commutator was due to Hyt\"onen \cite{Hy}.

 The characterisation of the compactness of the commutator $[b,\mathcal C]$ is also studied by using the pointwise kernel lower bound that we obtained. For the previous related result on compactness of commutators, we refer to (but not restricted to) \cite{U,CDLW, GWY,TYY,DLLWW}.

We now state our first main result.
\begin{theorem}\label{thm-cskernel}
The explicit formula of Cauchy--Szeg\"o kernel for the quaternionic Siegel upper half-space $ \mathcal U_n$ is given by
\begin{align*}
S(q,p)=s\Big(q_1+\overline p_1-2\sum_{k=2}^n\overline p_kq_k\Big),
\end{align*}
for $p=(p_1,p')=(p_1,\cdots, p_n)\in\mathcal U_n$, $q=(q_1,q')=(q_1,\cdots, q_n)\in\mathcal U_n$, where
\begin{align}\label{s 1}
s(\sigma)&= c_{n-1} {4 (2n-2)! \over |z|^4  (z-\bar z)^3 } {\bf i}\\
&\quad\times \bigg\{\operatorname{Im}\bigg[ {\bar z^2\over z^{2n-2}}\big( z+(2n-1){z-\bar z\over2} \big) \bigg] {\overline \sigma}-
\operatorname{Im} \bigg[ {\bar z^2\over z^{2n-3}}\big( z+(2n-2){z-\bar z\over2} \big)  \bigg]\bigg\},
\nonumber
\end{align}
here $
\sigma=x_1+x_2{\bf i}+x_3{\bf j}+x_4{\bf k}\in\mathbb H
$, $z=x_1+ |\operatorname{Im}\sigma|{\bf i}$, and $c_{n-1}$ is the one in Theorem A.
\end{theorem}

For any function $F:\mathcal U_n\rightarrow\mathbb H$, we write $F_\varepsilon$ for its ``vertical translate", where the vertical direction is given by the positive direction of $\operatorname{Re}q_1: F_\varepsilon(q)=F(q+\varepsilon{\bf e})$, where ${\bf e}=(1,0,0,\cdots,0)\in\mathbb H^{n}$. If $\varepsilon>0$, then $F_\varepsilon$ is defined in a neighborhood of $\partial \mathcal U_n$. In particular, $F_\varepsilon$ is defined on $\partial \mathcal U_n$.

The Cauchy--Szeg\"o projection operator $\mathcal C$ can be defined via the ``vertical translate"
from  Cauchy--Szeg\"o kernel for  $ \mathcal U_n$ by
\begin{align*}
(\mathcal C f)(q)=\lim_{\varepsilon\to 0}\int_{\partial \mathcal U_n} S(q+\varepsilon {\bf e}, p)f(p)d\beta(p),\quad
\forall f\in L^2(\partial \mathcal U_n), \quad q\in \partial \mathcal U_n,
\end{align*}
where the limit exists in the $L^2(\partial \mathcal U_n)$ norm and $\mathcal C(f)$ is the boundary limit of some function in $\mathcal H^2(\mathcal U_n)$.

In view of the action of the quaternionic Heisenberg group, the operator $\mathcal C$ can be explicitly described as a convolution operator on this group:
\begin{align}\label{cs projection gp}
(\mathcal C f)(g)=(f*K)(g)=p.v. \int_{\mathscr H^{n-1}}K(h^{-1}\cdot g)f(h)dh,
\end{align}
where the kernel $K(g)$ is defined  in Section 2 below.
We can write
\begin{align}\label{cs projection gp 1}
(\mathcal C f)(g)=p.v. \int_{\mathscr H^{n-1}}K(g,h)f(h)dh,
\end{align}
where $K(g,h)=K(h^{-1}\cdot g)\ {\rm for}\ g\neq h.$ Note that {\eqref{cs projection gp 1} holds whenever $f$ is an $L^2$ function supported in a compact set, for ever $g$ outside the support of $f$.}

As the second main result, we {now prove} the size and regularity estimate for the Cauchy--Szeg\"o kernel as follows.
\begin{theorem} \label{thm-ksize}
Suppose $j=1,\ldots, 4n-4,$ and we denote $Y_j$ the left-invariant  vector fields on $\mathscr H^{n-1}$ $($defined  as in \eqref{Y} in Section 2$)$. Then we have
\begin{align}\label{gradient}
\left|Y_jK(g)\right|\lesssim{1\over \rho(g,{\bf0})^{Q+1}},\quad g\in\mathscr H^{n-1}\setminus\{{\bf0}\},\quad
\end{align}
where ${\bf0}$ is the neutral element of $\mathscr H^{n-1}$.

Then we further have
the Cauchy--Szeg\"o kernel $K(g, h)$ on $\mathscr H^{n-1}$ $(g\not=h)$ satisfies the following conditions.
\begin{align*}
&{\rm (i)}\ \ |K(g, h)|\lesssim {1\over\rho(g,h)^{Q}};\\
&{\rm (ii)}\ \ |K(g,h) - K(g_0,h)|\lesssim   {\rho(g, g_0)\over \rho(g_0,h)^{Q+1} },\quad {\rm if}\ \rho(g_0,h)\geq c\rho(g,g_0);\\
&{\rm (iii)}\ \ |K(g,h) - K(g,h_0) |\lesssim   { \rho(h, h_0)\over \rho(g,h_0)^{Q+1} },\quad {\rm if}\ \rho(g,h_0)\geq c \rho(h,h_0)
\end{align*}
for some constant $c>0$,
where $Q=4n+2$ is the homogeneous dimension of $\mathscr H^{n-1}$ and $\rho$ is defined in Section 2.
\end{theorem}

As a consequence, we see that the Cauchy--Szeg\"o projection $\mathcal C$ is a standard Calder\'on--Zygmund operator, falling into the scope of the framework set forth in Chapter 1, Section 6.5 in \cite{St}.
Therefore, Theorem 3 in Chapter 1 in \cite{St} ($L^p$ boundedness, $1<p<\infty$) is applicable to $\mathcal C$. Moreover,
concerning the Hardy,  BMO  and VMO spaces studied in \cite{FS} on general homogeneous groups, we have the endpoint estimate for $\mathcal C$. And from \cite{CRW,U}, we also have the boundedness and compactness of the commutator of $\mathcal C$. We formulate these results as follows.
\begin{corollary}
Let all the notation be the same as above.
\begin{itemize}
\item[(1)] $\mathcal C$ extends to a  bounded operator on $L^p(\mathscr H^{n-1})$ for $1<p<\infty$;
\item[(2)] $\mathcal C$ is  of weak type $(1,1)$;
\item[(3)] $\mathcal C$ is  bounded from $H^1(\mathscr H^{n-1}) \to L^1(\mathscr H^{n-1})$;
\item[(4)] $\mathcal C$ is  bounded from $L^\infty(\mathscr H^{n-1}) \to {\rm BMO}(\mathscr H^{n-1})$;
\item[(5)] $[b,\mathcal C]$ is bounded on  $L^p(\mathscr H^{n-1})$ for $1<p<\infty$ if $b\in {\rm BMO}(\mathscr H^{n-1})$;
\item[(6)] $[b,\mathcal C]$ is compact on  $L^p(\mathscr H^{n-1})$ for $1<p<\infty$ if $b\in {\rm VMO}(\mathscr H^{n-1})$.
\end{itemize}
\end{corollary}

However, to obtain reverse arguments of (5) and (6) above, we still need to know more about the pointwise lower bound of the kernel $K$. To obtain the reverse of (5) we aim at verifying the kernel lower bound established by Hyt\"onen \cite{Hy}.
To obtain the reverse of (6) we aim at verifying  the same pointwise lower bound as in \cite{CDLW} in a twisted truncated sector.

\begin{theorem} \label{thm-lowerbd}
The Cauchy--Szeg\"o kernel $K(\cdot, \cdot)$ on $\mathscr H^{n-1}$ satisfies the following pointwise lower bound: there exist a large positive constant $r_0$ and a positive constant $C$ such that
for every $g\in \mathscr H^{n-1}$, there exists a ``twisted truncated sector'' $S_g\subset \mathscr H^{n-1}$ such that
$$ \inf_{g'\in S_g} \rho(g,g')=r_0 $$ and that for every $g_1\in B(g,1)$ and $g_2\in S_g$ we have
\begin{align*}
|K(g_1, g_2)|\geq  {C\over\rho(g_1,g_2)^Q}.
\end{align*}
Moreover, this sector $S_g$ is regular in the sense that $|S_g|=\infty$ and that for every
$R_2>R_1>2r_0$
$$ \big| \big(B(g,R_2)\backslash B(g,R_1)\big)  \cap S_g\big | \approx  \big| B(g,R_2)\backslash B(g,R_1)\big|$$
with the implicit constants  independent of $g$ and $R_1,R_2$.
\end{theorem}

\color{black}

By using Theorem \ref{thm-lowerbd}, we now establish the boundedness and compactness of the commutator of $\mathcal C$ with respect to
${\rm BMO}(\mathscr H^{n-1})$ and ${\rm VMO}(\mathscr H^{n-1})$, following the ideas and approaches in \cite{Hy} (see also \cite{DGKLWY}) and  in \cite{CDLW,DLLWW}, respectively.
\begin{theorem}\label{thm-commutator}
Suppose $1<p<\infty$ and $b\in L^1_{loc}(\mathscr H^{n-1})$.

{\rm(i)} $b\in {\rm BMO}(\mathscr H^{n-1})$ if and only if $[b,\mathcal C]$ is bounded on $L^p(\mathscr H^{n-1})$.

{\rm(ii)} $b\in {\rm VMO}(\mathscr H^{n-1})$ if and only if $[b,\mathcal C]$ is compact on $L^p(\mathscr H^{n-1})$.

\end{theorem}

Thus, we will only sketch the main approach of the proof of Theorem \ref{thm-commutator}.

In the next section we will provide some necessary preliminaries on quaternionic Heisenberg groups, and in the last section we provide the proofs of our main results.

\section{Preliminaries}



Recall that the space $\mathbb H$ of quaternion numbers forms a division algebra with respect to the coordinate addition and the quaternion multiplication
\begin{align*}
x x'&=(x_1+x_2{\bf i}+x_3{\bf j}+x_4{\bf k})(x'_1+x'_2{\bf i}+x'_3{\bf j}+x'_4{\bf k})\\
&=x_1x'_1-x_2x'_2-x_3x'_3-x_4x'_4
+\left(x_1x'_2+x_2x'_1+x_3x'_4-x_4x'_3 \right){\bf i}\\
&\quad +\left(x_1x'_3-x_2x'_4+x_3x'_1+x_4x'_2 \right){\bf j}
+\left(x_1x'_4+x_2x'_3-x_3x'_2+x_4x'_1 \right){\bf k},
\end{align*}
for any $x=x_1+x_2{\bf i}+x_3{\bf j}+x_4{\bf k}$, $x'=x'_1+x'_2{\bf i}+x'_3{\bf j}+x'_4{\bf k}\in \mathbb H$.
The conjugate $\bar x$ is defined by
$$\bar x=x_1-x_2{\bf i}-x_3{\bf j}-x_4{\bf k},$$
 and the modulus $|x|$ is defined by
$$|x|^2=x \bar x=\sum_{j=1}^4x_j^2.$$
 The conjugation inverses the product of quaternion number in the following sense $\overline{q \sigma}=\bar \sigma \bar{q}$ for any $q, \sigma\in\mathbb H$.
It is clear that
\begin{equation}\begin{split}\label{im}
\operatorname{Im}(\bar{x}x')&=\operatorname{Im}\{(x_1-x_2{\bf i}-x_3{\bf j}-x_4{\bf k})(x'_1+x'_2{\bf i}+x'_3{\bf j}+x'_4{\bf k})\}\\
&=\left(x_1x'_2-x_2x'_1-x_3x'_4+x_4x'_3 \right){\bf i}
+\left(x_1x'_3+x_2x'_4-x_3x'_1-x_4x'_2 \right){\bf j}\\
&\quad
+\left(x_1x'_4-x_2x'_3+x_3x'_2-x_4x'_1 \right){\bf k}\\
&=:\sum_{\alpha=1}^3\sum_{k,j=1}^4b_{kj}^\alpha x_kx'_j{\bf i}_\alpha,
\end{split}
\end{equation}
where ${\bf i}_1={\bf i}, {\bf i}_2={\bf j}, {\bf i}_3={\bf k}$, and $b_{kj}^\alpha$ is the $(k,j)$ th entry of the following matrices $b^\alpha$:
\begin{equation*}
b^1:=\left( \begin{array}{cccc}
0&1 & 0 & 0\\
-1 & 0 & 0&0\\
0 & 0 & 0&-1\\
0 & 0 & 1&0
\end{array}
\right ),
\quad
b^2:=\left( \begin{array}{cccc}
0&0 & 1 & 0\\
0 & 0 & 0&1\\
-1 & 0 & 0&0\\
0 & -1 & 0 &0
\end{array}
\right ),
\quad
b^3:=\left( \begin{array}{cccc}
0&0 & 0 & 1\\
0 & 0 & -1&0\\
0 & 1 & 0&0\\
-1 & 0 & 0 &0
\end{array}
\right ).
\end{equation*}

The quaternion space $\mathbb H^n$ is isomorphic to $\mathbb R^{4n}$ as a real vector space and the pure imaginary $\operatorname{Im}\mathbb H$ are isomorphic to $\mathbb R^3$.

 The quaternionic Heisenberg group $\mathscr H^{n-1}$  is the space $\mathbb R^{4n-1}=\mathbb R^3\times\mathbb R^{4(n-1)}$, which is isomorphic to $\operatorname{Im}\mathbb H\times \mathbb H^{n-1}$ endowed with the non-commutative product
\begin{align}\label{prod}
(t,y)\cdot (t',y')=\left(t+t'+2\operatorname{Im}\langle y, y'\rangle, y+y'\right),
\end{align}
where $t=t_1{\bf i}+t_2{\bf j}+t_3{\bf k}$, $t'=t'_1{\bf i}+t'_2{\bf j}+t'_3{\bf k}\in \operatorname{Im}\mathbb H$, $y, y'\in\mathbb H^{n-1}$, and $\langle \cdot, \cdot\rangle$ is the inner product defined by
$$\langle y, y'\rangle=\sum_{l=1}^{n-1}\overline y_l y'_l, \quad y=(y_1,\cdots, y_{n-1}),\quad y'=(y'_1,\cdots, y'_{n-1})\in\mathbb H^{n-1}.$$
It is easy to check that the identity of $\mathscr H^{n-1}$ is the origin ${\bf 0}:=(0,0)$, and the inverse of $(t, y)$ is given by $(-t, -y)$.

The boundary of quaternionic Siegel upper half-space $\partial \mathcal U_n$ can be identified with
 the quaternionic Heisenberg group $\mathscr H^{n-1}$  via  the projection
\begin{align}\label{p}
\pi: \partial\mathcal U_n &\longrightarrow \operatorname{Im}\mathbb H\times\mathbb H^{n-1},\\
\nonumber(|q'|^2+x_2{\bf i}+x_3{\bf j}+x_4{\bf k}, q')&\longmapsto (x_2{\bf i}+x_3{\bf j}+x_4{\bf k}, q').
\end{align}
Let $d\beta$ be the Lebesgue measure on $\partial\mathcal U_n$ obtained by pulling back the Haar measure on the group $\mathscr H^{n-1}$ by the projection $\pi$.



The Hardy space $\mathcal H^2(\mathcal U_n)$ consists of all regular functions $F$ on $\mathcal U_n$, for which
$$\|F\|_{\mathcal H^2(\mathcal U_n)}:=\Big(\sup_{\varepsilon>0}\int_{\partial \mathcal U_n}|F_\varepsilon(q)|^2d\beta(q)\Big)^{1\over 2}<\infty.$$
According to \cite[Theorem 4.1]{CMW}, a function $F\in \mathcal H^2(\mathcal U_n)$ has boundary value $F^b$ that belongs to $L^2(\mathcal U_n)$.

By \eqref{im}, the multiplication of the quaternionic Heisenberg group in terms of real variables can be written as (cf. \cite{shi-wang})
\begin{align}\label{law}
(t,y)\cdot(t', y')=\bigg(t_\alpha+t'_\alpha+2\sum_{l=0}^{n-1}\sum_{j,k=1}^4b_{kj}^\alpha y_{4l+k}y'_{4l+j}, y+y' \bigg),
\end{align}
where
$t=(t_1, t_2, t_3)$, $t'=(t'_1, t'_2, t'_3)\in\mathbb R^3$, $\alpha=1,2,3$, $y=(y_1,y_2,\cdots, y_{4n-4})$, $y'=(y'_1,y'_2,\cdots, y'_{4n-4})\in\mathbb R^{4n-4}$.

 The following vector fields are left invariant on the quaternionic Heisenberg group by the multiplication laws of the quaternionic Heisenberg group in \eqref{law}:
\begin{align}\label{Y}
 Y_{4l+j} ={\partial\over \partial y_{4l+j}} + 2 \sum_{\alpha=1}^3\sum_{k=1}^4 b_{kj}^\alpha y_{4l+k} {\partial\over\partial t_\alpha},
 \end{align}
and
 $$[Y_{4l'+k}, Y_{4l+j}]=2\delta_{ll'}\sum_{\alpha=1}^3b_{kj}^\alpha{\partial\over \partial t_\alpha},$$
for $l, l'=0,\cdots,n-2$, $j,k=1,\cdots 4$.
 Then the horizontal tangent space at $g\in\mathscr H^{n-1}$, denoted by $H_g$, is spanned by the left invariant vectors $Y_1(g),\cdots, Y_{4n-4}(g)$. For each $g\in\mathscr H^{n-1}$, we fix a quadratic form
 $\langle \cdot ,\cdot\rangle_H$ on $H_g$ with respect to which the vectors $Y_1(g),\cdots, Y_{4n-4}(g)$
 are orthonormal.


For any $p=(t,y)\in \mathscr H^{n-1}$, we can associate the automorphism $\tau_p$ of $\mathcal U_n$:
\begin{align}\label{tau h}
\tau_p: (q_1,q')\longmapsto\left(q_1+|y|^2+t+2\langle y, q' \rangle, q'+y\right).
\end{align}
It is obviously extended to the boundary $\partial\mathcal U_n$. It is easy to see that the action on $\partial \mathcal U_n$ is transitive. In particular, we have
$$\tau_p: (0,0)\longmapsto (|y|^2+t, y).$$
 And we can write each $q\in\partial \mathcal U_n$ as $q=\tau_g(0)$ for a unique $g\in \mathscr H^{n-1}$. In this correspondence we have that
{$d\beta(q)=dg$}, the invariant measure on $\mathscr H^{n-1}$. Similarly, we write $p\in\partial \mathcal U_n$ in the form $p=\tau_h(0)$ for some unique $h\in\mathscr H^{n-1}$. Then from \eqref{tau h} we can see that
 \begin{align*}
 S(q+\varepsilon {\bf e}, p)&=S\big(\tau_{h^{-1}}(q+\varepsilon {\bf e}), \tau_{h^{-1}}(p)\big)
 =S\big(\tau_{h^{-1}}(q)+\varepsilon {\bf e}, 0\big)
 =S\big(\tau_{h^{-1}\cdot g}(0)+\varepsilon {\bf e}, 0\big).
 \end{align*}
 Take $K_\varepsilon(g)=S(\tau_g(0)+\varepsilon {\bf e},0)$, and denote by
 {$f(h):=f(\tau_h(0))=f(p)$   and
 $(\mathcal C f)(g):=(\mathcal C f)(\tau_g(0))=(\mathcal C f)(q)$ by abuse of notations}.
 Then
 \begin{align}\label{kepsilon}
 (\mathcal C f)(g)=\lim_{\varepsilon\to 0}\int_{\mathscr H^{n-1}}K_{\varepsilon}(h^{-1}\cdot g)f(h)dh,\quad f\in L^2(\mathscr H^{n-1}),
 \end{align}
 where the limit is taken in $L^2(\mathscr H^{n-1})$.

 Recall that the convolution on $\mathscr H^{n-1}$ is defined as
 $$(f*\tilde f)(g)=\int_{\mathscr H^{n-1}}f(h)\tilde f(h^{-1}\cdot g)dh.$$
Therefore, \eqref {kepsilon} can be formally rewritten as
$$ (\mathcal C f)(g)=(f*K)(g),$$
where $K$ is the distribution given by
$\lim_{\varepsilon\to 0}K_{\varepsilon}.$
Thus, if $g=(t, y)\in \mathscr H^{n-1}$ with $t=t_1{\bf i}+t_2{\bf j}+t_3{\bf k}$, then
\begin{align}\label{kg}
K(g)=\lim_{\varepsilon\to 0}K_{\varepsilon}(g)=s(|y|^2+t).
\end{align}


For any $g=(t, y)\in \mathscr H^{n-1}$, the homogeneous norm of $g$ is defined by
$$\|g\|=\left(|y|^4+\sum_{j=1}^3|t_j|^2\right)^{1\over 4}.$$
Obviously,
$\|g^{-1}\|=\|-g\|=\|g\|$ and $\|\delta_r(g)\|=r\|g\|$, where $\delta_r, r>0,$ is the dilation on $\mathscr H^{n-1}$, which is defined as
$$\delta_r(t,y)=(r^2t, ry).$$

On $\mathscr H^{n-1}$, we define the quasi-distance $$\rho(h,g)=\|g^{-1}\cdot h\|.$$ It is clear that
 $\rho$ is symmetric and satisfies the generalized triangle inequality
\begin{align}\label{rho}
\rho(h,g)\leq C_\rho(\rho(h,w)+\rho(w,g)),
\end{align}
for any $h, g, w\in \mathscr H^{n-1}$ and some $C_\rho>0$.
 Using $\rho$, we define the balls $B(g,r)$ in $\mathscr H^{n-1}$ by $B(g,r)=\{h: \rho(h,g)<r \}$. Then
$$|B(g,r)|\approx r^Q.$$

We now recall the Hardy, BMO and VMO spaces.
Note that $\mathscr H^{n-1}$ falls into the scope of homogeneous group, and hence we have the natural  BMO space (also the Hardy space) in this setting due to Folland and Stein \cite{FS}. To be self-enclosed, we recall the definition of the BMO space.
$${\rm BMO}(\mathscr H^{n-1})=\{ b\in L^1_{loc}(\mathscr H^{n-1}): \|b\|_{ {\rm BMO}(\mathscr H^{n-1}) }<\infty \},$$
where
$$ \|b\|_{ {\rm BMO}(\mathscr H^{n-1}) }= \sup\limits_{B} {1\over |B|} \int_{B} | b(g)-b_B |dg,$$
 where the supremum is taken over all balls $B\subset \mathscr H^{n-1}$, and $b_B$ is the average of $b$ over the ball $B$. Similarly we also have the VMO space on $\mathscr H^{n-1}$, which is closure of the $C^\infty_c(\mathscr H^{n-1})$ under the norm of $\|\cdot\|_{ {\rm BMO}(\mathscr H^{n-1}) }$, see \cite{CDLW} for more details of the definition and properties of this VMO space in the more general setting, the stratified Lie group.


\section{Proof of main results}

\begin{proof}[Proof of Theorem \ref{thm-cskernel}]

Note that for $\sigma =  x_1+x_2{\bf i}+x_3{\bf j}+x_4{\bf k} $, we have
 $|\sigma|^2 =x_1^2+x_2^2+x_3^2+x_4^2$, and
$\bar \sigma =  x_1-x_2{\bf i}-x_3{\bf j}-x_4{\bf k} $. By Theorem A,
\begin{align}\label{s sigma}
s(\sigma)&=c_{n-1}{\partial^{2(n-1)}\over \partial x_1^{2(n-1)}}{{\overline \sigma}\over |\sigma|^4}\\
&= c_{n-1}{\partial^{2(n-1)}\over \partial x_1^{2(n-1)}}\bigg({ 1\over |\sigma|^4} \bigg)  {\overline \sigma} +
{(2n-2)}c_{n-1}{\partial^{2n-3}\over \partial x_1^{2n-3}}\bigg({ 1\over |\sigma|^4}\bigg) .\nonumber
\end{align}

Now it suffices to figure out the representation of
$$ {\partial^{l}\over \partial x_1^{l}}\bigg({ 1\over |\sigma|^4} \bigg),  $$
where $l$ is any fixed natural number.

To see this, let $z =x_1+ |\operatorname{Im}\sigma|{\bf i} $, and then
$\bar z = x_1- |\operatorname{Im}\sigma|{\bf i}$. Hence, we further have
$$  |\sigma|^2 = z \bar z.$$

As a consequence, we have
\begin{align*}
{\partial^{l}\over \partial x_1^{l}}\bigg({ 1\over |\sigma|^4} \bigg) = {\partial^{l}\over \partial x_1^{l}}\bigg({ 1\over z^2 \bar z^2} \bigg).
\end{align*}
By using the binomial expansion, we obtain that
\begin{align*}
{\partial^{l}\over \partial x_1^{l}}\bigg({ 1\over |\sigma|^4} \bigg) &=
\sum_{k=0}^l C_l^k {\partial^{l-k}\over \partial x_1^{l-k}} \bigg( {1\over z^2} \bigg)
{\partial^{k}\over \partial x_1^{k}} \bigg( {1\over \bar z^2} \bigg),
\end{align*}
where
$$ C_l^k = { l! \over k! (l-k)!}. $$
Next, by collating the coefficients, we further have
\begin{align}\label{derivative}
{\partial^{l}\over \partial x_1^{l}}\bigg({ 1\over |\sigma|^4} \bigg)
&=\sum_{k=0}^l   (-1)^l  l! (l-k+1)(k+1)   \bigg( {1\over z^{l-k+2}} \bigg)
 \bigg( {1\over \bar z^{k+2}} \bigg).
\end{align}
Then the key step is to obtain an explicit formula for the above summation.
To see this, we rewrite the right-hand side of \eqref{derivative} as follows.
\begin{align*}
{\rm RHS\ of\ } \eqref{derivative}
 &={ (-1)^l l! \over z^l |z|^4 }\sum_{k=0}^l    (l-k+1)(k+1)
 \bigg( {z\over \bar z} \bigg)^{k}
 =:{ (-1)^l l! \over z^l |z|^4 } \ {\bf A} \bigg( {z\over \bar z} \bigg).
\end{align*}
Now let $w = {z\over \bar z}$, then
\begin{align*}
 {\bf A}(w) &=\sum_{k=0}^l    (l-k+1)(k+1)
w^{k}.
\end{align*}
Based on the definition of $ {\bf A}(w) $, we now consider the following function
$$ {\bf G}(w):= \sum_{k=0}^l    (l-k+1)w^{k+1}.$$
It is obvious that $${d\over dw} {\bf G}(w) ={\bf A}(w). $$
To continue, we now set $\zeta = w^{-1}$. Then by changing of variable we have
$$  {\bf G}\Big({1\over\zeta}\Big) =  \sum_{k=0}^l    (l-k+1)\zeta^{-k-1}. $$
Next, we set
$$ {\bf H} (\zeta):= \zeta^{l+1}  {\bf G}\Big({1\over\zeta}\Big) =\sum_{k=0}^l    (l-k+1)\zeta^{l-k}.$$
Again, based on the definition of ${\bf H}$, we now consider the following function
$${\bf F}(\zeta):=  \sum_{k=0}^l    \zeta^{l-k+1} ,$$
and it is obvious that
$$ {d\over d\zeta}{\bf F}(\zeta) = {\bf H}(\zeta).$$

Now by taking the summation we have
\begin{align*}
{\bf F}(\zeta)=  \zeta^{l+1} \sum_{k=0}^l    \zeta^{-k}= { \zeta (\zeta^{l+1}-1)\over \zeta-1  },
\end{align*}
and hence
\begin{align*}
{\bf H}(\zeta)&={d\over d\zeta} \bigg( { \zeta (\zeta^{l+1}-1)\over \zeta-1  }\bigg)
 = { (l+1)\zeta^{l+2} -(l+2)\zeta^{l+1} +1 \over (\zeta-1)^2  }.
\end{align*}
As a consequence, we get that
\begin{align*}
{\bf G}(w)&= { (l+1)w\over (1-w)^2  } -  { (l+2)w^2\over (1-w)^2  } +  { w^{l+3}\over (1-w)^2  }.
\end{align*}
We now have
\begin{align*}
{\bf A}(w)&= {d\over dw}\bigg({ (l+1)w\over (1-w)^2  } -  { (l+2)w^2\over (1-w)^2  } +  { w^{l+3}\over (1-w)^2  }\bigg)\\
&={ (l+1)-(l+3) w + (l+3)w^{l+2} -(l+1)w^{l+3} \over (1-w)^3}.
\end{align*}
By changing $w$ back to ${z\over \bar z}$, we have
\begin{align*}
{\bf A}\Big({z\over \bar z}\Big)&=
{ (l+1) {z^{l+3} \over \bar z^l}- (l+3){ z^{l+2} \over \bar z^{l-1} }  + (l+3) z\, \bar z^2 -(l+1) \bar z^3  \over ( z-\bar z )^3   }
\end{align*}

Hence,
\begin{align*}
{\rm RHS\ of\ } \eqref{derivative}
 &={ (-1)^l\ l! \over |z|^4 } \ \, { (l+1) {z^{3} \over \bar z^l}- (l+3){ z^{2} \over \bar z^{l-1} }  + (l+3){ \bar z^2 \over z^{l-1}} -(l+1) {\bar z^3 \over z^l }\over ( z-\bar z )^3   }\\
 &={ (-1)^l\ l! \over |z|^4\ ( z-\bar z )^3 } \ \bigg( - 2(l+1) {\rm Im}\Big( {\bar z^{3} \over  z^l}\Big)\, {\bf i}+ 2(l+3) {\rm Im}\Big({ \bar z^{2} \over  z^{l-1} }\Big)\, {\bf i} \bigg)\\
 &= {4 (-1)^l \ l! \over |z|^4  (z-\bar z)^3 }\operatorname{Im}\bigg[ {\bar z^2\over z^l}\big( z+(l+1){z-\bar z\over2} \big) \bigg]{\bf i}.
\end{align*}

\medskip


Based on the representation in \eqref{derivative} above, from \eqref{s sigma}, we obtain that
\begin{align}\label{s sigma1}
s(\sigma)&= c_{n-1} {4  (2n-2)! \over |z|^4    (z-\bar z)^3 }\operatorname{Im}\bigg[ {\bar z^2\over z^{2n-2}}\big( z+ (2n-1){z-\bar z\over2} \big)  \bigg] {\bf i} {\overline \sigma} \nonumber\\
&\quad-
{(2n-2)}  c_{n-1} {4  (2n-3)! \over |z|^4    (z-\bar z)^3 }\operatorname{Im} \bigg[ {\bar z^2\over z^{2n-3}}\big( z+ (2n-2){z-\bar z\over2} \big) \bigg]  {\bf i}\\
&= c_{n-1} {4  (2n-2)! \over |z|^4    (z-\bar z)^3 }\bigg\{\operatorname{Im}\bigg[ {\bar z^2\over z^{2n-2}}\big( z+ (2n-1){z-\bar z\over2} \big)  \bigg] {\bf i} {\overline \sigma}\nonumber\\
&\qquad-\operatorname{Im} \bigg[ {\bar z^2\over z^{2n-3}}\big( z+ (2n-2){z-\bar z\over2} \big) \bigg]  {\bf i}\bigg\}.  \nonumber
\end{align}
The proof of Theorem \ref{thm-cskernel} is complete.
\end{proof}

 Based on the result of Theorem  \ref{thm-cskernel} we see that the
 kernel $K$ of the Cauchy--Szeg\"o projection operator $\mathcal C$ on $\mathscr H^{n-1}$
 is homogeneous. To be more specific, we have the following.
 \begin{corollary}\label{cor homo}
 The kernel $K$ of the Cauchy--Szeg\"o projection operator $\mathcal C$ on $\mathscr H^{n-1}$
 satisfies
 $$ K( \delta_r (g) ) = r^{-Q} K(g) $$
 for every $g\in \mathscr H^{n-1}$ and $r>0$.
 \end{corollary}

Before proving  Theorem \ref{thm-ksize}, we will need
a mean value theorem on stratified groups. Note that there is a version established in \cite[(1.41)]{FS}. However it requires the function $f\in C^1(\mathscr H^{n-1})$, while the kernel $K$ of the Cauchy--Szeg\"o projection operator $\mathcal C$ has singularity at the origin.

We now provide the following mean value theorem on stratified groups with respect to the kernel $K$.
To begin with we recall
that an absolutely continuous curve $\gamma: [0,1]\to \mathscr H^{n-1}$ is horizontal if its tangent vectors
$\dot{\gamma}(t),  t\in[0, 1]$, lie in the horizontal tangent space $H_{\gamma(t)}$. By \cite{Chow}, any given two points $p,q\in \mathscr H^{n-1}$ can be connected by a horizontal curve.

 The Carnot--Carath\'eodory metric on $\mathscr H^{n-1}$ as follows. For $g,h\in \mathscr H^{n-1}$,
\begin{align}\label{cc metric}
d_{cc}(g,h):= \inf_\gamma \int_0^1 \langle  \dot{\gamma}(t), \dot{\gamma}(t)\rangle_H^{1\over2}\ dt,
\end{align}
where $\gamma: [0,1]\to \mathscr H^{n-1}$ is a horizontal Lipschitz curve with $\gamma(0)=g$, $\gamma(1)=h$.
It is known that the Carnot--Carath\'eodory metric $d_{cc}$ is left-invariant,  and it is equivalent to the homogeneous metric $\rho$ in the sense that: there exist $C_d, \tilde C_d>0$ such that for any $g, h\in\mathscr{H}^{n-1}$ (see \cite{BLU}),
\begin{align}\label{norm-equ}
C_d\rho(g,h)\leq d_{cc}(g,h)\leq \tilde C_d\rho(g,h).
 \end{align}

\begin{lemma}\label{lem-mean}
There exist $C>0$ and $c_0\in(0,1)$ such that
for $g,g_0\in \mathscr H^{n-1}\backslash\{{\bf 0}\}$ with
$d_{cc}(g,g_0)<c_0 d_{cc}(g_0, {\bf0})$,
we have
$$ |K(g)-K(g_0)|\leq Cd_{cc}(g,g_0)\times \max_{\substack{1\leq j\leq 4n-4\\ u: d_{cc}(u,0)\leq d_{cc}(g, g_0)}} \big|Y_j K(g_0\cdot u)\big|. $$
\end{lemma}
\begin{proof}
Suppose that $g,g_0\in \mathscr H^{n-1}\backslash\{{\bf 0}\}$ with
$d_{cc}(g,g_0)<c_0 d_{cc}(g_0, {\bf 0})$, and let $\gamma$ be the curve in the definition of \eqref{cc metric}
with $\gamma(0)=g$ and $\gamma(1)=g_0$ such that
$$ d_{cc}(g,g_0)\leq \int_0^1 \langle \dot{\gamma}(t),  \dot{\gamma}(t)\rangle_H^{1\over2} \ dt \leq { 101\over 100 } d_{cc}(g,g_0).$$
Then for any $t\in [0,1]$, we have
\begin{align*}
d_{cc}(\gamma(t), {\bf 0})&\geq d_{cc}(g_0, {\bf 0}) -d_{cc}(\gamma(t),g_0) \geq d_{cc}(g_0, {\bf 0}) -c_0d_{cc}(g_0, {\bf 0}) \\
&=(1-c_0) d_{cc}(g_0, {\bf 0}).
\end{align*}
So $K(\gamma(t))$ is not singular for every $t\in [0,1]$.

To continue, we now write $\gamma(t)=(\gamma_1(t),\ldots,\gamma_{4n-1}(t)  )\in \mathscr{H}^{n-1}$.
Then $$\dot{\gamma}(t) = \sum_{j=1}^{4n-4} \gamma'_j(t) {\partial\over \partial y_j}
+ \sum_{\alpha=1}^3 \gamma'_{4n-4+\alpha}(t) {\partial\over \partial t_\alpha} = \sum_{j=1}^{4n-4} a_j(t) Y_j\big( \gamma(t)),    $$
where
$a_j(t):=\gamma'_j(t) $ for $j=1,2,\ldots,4n-4$,
$$ \gamma'_{4n-4+\alpha}(t) = 2\sum_{l=1}^{n-2}\sum_{j=1}^4 \sum_{k=1}^4 a_{4l+j} b_{kj}^\alpha \gamma_{4l+k}(t), $$
and $\langle \dot{\gamma}(t), \dot{\gamma}(t)\rangle_H = \sum_{j=1}^{4n-4} |a_j(t)|^2 $.

As a consequence, we have
\begin{align*}
K(g)-K(g_0) &= \int_0^1 {d\over dt} \Big( K(\gamma(t)) \Big) dt = \int_0^1 \Bigg( \sum_{j=1}^{4n-4} {\partial K\over \partial y_j} \gamma'_j(t) + \sum_{\alpha=1}^3 {\partial K\over \partial t_\alpha} \gamma'_{4n-4+\alpha}(t)\Bigg)  dt\\
&=  \int_0^1  \sum_{j=1}^{4n-4} a_j(t)\,Y_jK(\gamma(t))\ dt.
\end{align*}
Hence,
\begin{align*}
|K(g)-K(g_0)|
&\leq \max_{\substack{ 1\leq j\leq 4n-4 \\ 0\leq t\leq 1 }}|Y_jK(\gamma(t))|\     \int_0^1 \sum_{j=1}^{4n-4} |a_j(t)|\, dt\\
&\leq C\max_{\substack{ 1\leq j\leq 4n-4 \\ 0\leq t\leq 1 }}|Y_jK(\gamma(t))|\     \int_0^1 \bigg(\sum_{j=1}^{4n-4} |a_j(t)|^2\bigg)^{1\over 2}\, dt\\
&= C\max_{\substack{ 1\leq j\leq 4n-4 \\ 0\leq t\leq 1 }}|Y_jK(\gamma(t))|\     \int_0^1 \langle \dot{\gamma}(t), \dot{\gamma}(t)\rangle^{1\over 2}\, dt\\
&\leq C\max_{\substack{ 1\leq j\leq 4n-4 \\ 0\leq t\leq 1 }}|Y_jK(\gamma(t))| d_{cc}(g,g_0)\\
&\leq Cd_{cc}(g,g_0)\  \max_{\substack{1\leq j\leq 4n-4\\ u: d_{cc}(u,0)\leq d_{cc}(g, g_0)}} \big|Y_j K(g_0\cdot u)\big|,
\end{align*}
where $C$ is an absolute constant depending only on $n$.

The proof of Lemma \ref{lem-mean} is complete.
\end{proof}

\begin{proof}[Proof of Theorem \ref{thm-ksize}]

We begin with proving
the size estimate in (i).

 By \eqref{kg}, for any $g=(t, y)\in \mathscr H^{n-1}$, where $t=t_1{\bf i}+t_2{\bf j}+t_3{\bf k}$, $y\in\mathbb H^{n-1}$, we have
$$K(g)=s(|y|^2+t).$$

From \eqref{s sigma} and \eqref {derivative}, we can see that for $\sigma=x_1+x_2{\bf i}+x_3{\bf j}+x_4{\bf k}$,
\begin{align}\label{ssigma1}
s(\sigma)
&= c_{n-1}\sum_{k=0}^{2n-2}   (2n-2)!  (2n-k-1) (k+1)    {\bar\sigma\over z^{2n-k}   \bar z^{k+2}} \\
&\quad -  c_{n-1}\sum_{k=0}^{2n-3}   (2n-2)!  (2n-k-2) (k+1)   {1\over z^{2n-k-1}  \bar z^{k+2}},\nonumber
\end{align}
where $z =x_1+|\operatorname{Im}\sigma|{\bf i} $.
Therefore,
\begin{align*}
K(g)&=s(|y|^2+t)\\
&=c_{n-1}\sum_{k=0}^{2n-2}   (2n-2)!  (2n-k-1) (k+1){|y|^2-t\over \left(|y|^2+|t|{\bf i}\right)^{2n-k}\left(|y|^2-|t|{\bf i}\right)^{k+2}} \\
&\quad - c_{n-1}\sum_{k=0}^{2n-3} (2n-2)!  (2n-k-2) (k+1) {1\over \left(|y|^2+|t|{\bf i}\right)^{2n-k-1}\left(|y|^2-|t|{\bf i}\right)^{k+2}}.
\end{align*}

Since
\begin{align*}
\bigg|{|y|^2-t\over \left(|y|^2+|t|{\bf i}\right)^{2n-k}\left(|y|^2-|t|{\bf i}\right)^{k+2}} \bigg|
&={\big(|y|^4+|t|^2 \big)^{1\over 2}\over\big(|y|^4+|t|^2 \big)^{2n-k\over 2}\big(|y|^4+|t|^2 \big)^{k+2\over 2}}\\
&={1\over \big(|y|^4+|t|^2 \big)^{2n+1 \over 2}}={1\over \|g\|^{Q}},
\end{align*}
and
\begin{align*}
\bigg|{1\over \left(|y|^2+|t|{\bf i}\right)^{2n-k-1}\left(|y|^2-|t|{\bf i}\right)^{k+2}} \bigg|
={1\over \big(|y|^4+|t|^2 \big)^{2n+1 \over 2}}={1\over \|g\|^{Q}}.
\end{align*}
We have
$$|K(g)|\lesssim {1\over \|g\|^{Q}},$$
where the implicit constant depends only on $n$,
which shows that (i) holds.

\medskip
We now prove the regularity estimate as in
(ii).

To begin with, we recall that the vector field
$$ Y_{4l+j} ={\partial\over \partial y_{4l+j}} + 2 \sum_{\alpha=1}^3\sum_{k=1}^4 b_{kj}^\alpha y_{4l+k} {\partial\over\partial t_\alpha},\quad l=0,\cdots,n-2, \quad j=1,\cdots 4,$$
where for $\alpha=1,2,3$, $b_{kj}^\alpha$ is the element (from the $k$th row and the $j$th column)
in the matrix $b^\alpha$ below:
\begin{equation*}
b^1:=\left( \begin{array}{cccc}
0&1 & 0 & 0\\
-1 & 0 & 0&0\\
0 & 0 & 0&-1\\
0 & 0 & 1&0
\end{array}
\right ),
\quad
b^2:=\left( \begin{array}{cccc}
0&0 & 1 & 0\\
0 & 0 & 0&1\\
-1 & 0 & 0&0\\
0 & -1 & 0 &0
\end{array}
\right ),
\quad
b^3:=\left( \begin{array}{cccc}
0&0 & 0 & 1\\
0 & 0 & -1&0\\
0 & 1 & 0&0\\
-1 & 0 & 0 &0
\end{array}
\right ).
\end{equation*}

%
For any $g,h,g_0 $ with  $\rho(g_0,h)\geq c\rho(g,g_0)$, we set $x = h^{-1}\cdot g_0$, then by Lemma \ref{lem-mean} and \eqref{norm-equ}, we have
\begin{equation}\begin{split}\label{kk}
|K(g, h)-K(g_0,h)|&=\left|K(h^{-1}\cdot g)-K(h^{-1}\cdot g_0)\right|
\\
&=\left|K(x\cdot (g_0^{-1}\cdot g))-K(x)\right|\\
&\leq C d_{cc}(g, g_0) \max_{\substack{1\leq j\leq 4n-4\\ u: d_{cc}(u,0)\leq d_{cc}(g, g_0)}} \big|Y_j K(x\cdot u)\big|\\
&\leq C \rho(  g,g_0 )
\max_{\substack{1\leq j\leq 4n-4\\ u: d_{cc}(u,0)\leq d_{cc}(g, g_0)}} \big|Y_j K(x\cdot u)\big|.
\end{split}\end{equation}

Recall that  for any $g=(t,y) = (t_1{\bf i}+t_2 {\bf j}+ t_3 {\bf k}, y)$,
$$ K(g) = s( |y|^2 +t ) = s( |y|^2, t_1,t_2 , t_3 ). $$
Then for any $l=0,\cdots,n-2$, $ j=1,\cdots 4$.
\begin{align*}
 Y_{4l+j} K(g)& = \bigg( {\partial\over \partial y_{4l+j}} + 2 \sum_{\alpha=1}^3\sum_{k=1}^4 b_{kj}^\alpha y_{4l+k} {\partial\over\partial t_\alpha} \bigg) s( |y|^2, t_1,t_2 , t_3 )\\
& = \bigg( {\partial\over \partial y_{4l+j}} + 2 \sum_{\alpha=1}^3\sum_{k=1}^4 b_{kj}^\alpha y_{4l+k} {\partial\over\partial t_\alpha}  \bigg) s( |y|^2, t_1,t_2 , t_3 )\\
&=2y_{4l+j} {\partial s\over \partial x_1}( |y|^2, t_1,t_2 , t_3 )  +2 \sum_{\alpha=1}^3\sum_{k=1}^{4} b_{kj}^\alpha y_{4l+k} {\partial s\over \partial x_{\alpha+1}}( |y|^2, t_1,t_2 , t_3 )
\end{align*}

From \eqref {s sigma} and \eqref {derivative}, we can see that
%
%
\begin{align*}
 {\partial s\over \partial x_1}&= c_{n-1} { \partial ^{2n-1}\over \partial x_1^{2n-1}  }\bigg( {1\over |\sigma|^4} \bigg) \bar\sigma + c_{n-1} { \partial ^{2n-2}\over \partial x_1^{2n-2}  }\bigg( {1\over |\sigma|^4} \bigg)\,  {\partial  \bar\sigma\over \partial x_1}
 + (2n-2)c_{n-1} { \partial ^{2n-2}\over \partial x_1^{2n-2}  }\bigg( {1\over |\sigma|^4} \bigg)\, \\
 &= c_{n-1} { \partial ^{2n-1}\over \partial x_1^{2n-1}  }\bigg( {1\over |\sigma|^4} \bigg) \bar\sigma + (2n-1)c_{n-1} { \partial ^{2n-2}\over \partial x_1^{2n-2}  }\bigg( {1\over |\sigma|^4} \bigg)\\
 &=-c_{n-1}\sum_{k=0}^{2n-1}     (2n-1)! (2n-k)(k+1)   {1\over z^{2n-k+1}\bar z^{k+2}}
\bar\sigma\\
 &\quad+c_{n-1}\sum_{k=0}^{2n-2}     (2n-1)! (2n-k-1)(k+1)   {1\over z^{2n-k}\bar z^{k+2}},
\end{align*}
where we used the fact that $ {\partial  \bar\sigma\over \partial x_1} =1 $.

Next, by \eqref{ssigma1}, we have
\begin{align*}
{\partial s\over \partial x_{\alpha+1}}
&= c_{n-1}\sum_{k=0}^{2n-2}   (2n-2)!  (2n-k-1) (k+1)  {\partial \over \partial x_{\alpha+1}} \bigg[ {\bar\sigma\over z^{2n-k}   \bar z^{k+2}} \bigg]\\
&\quad -  c_{n-1}\sum_{k=0}^{2n-3}   (2n-2)!  (2n-k-2) (k+1)  {\partial \over \partial x_{\alpha+1}} \bigg[ {1\over z^{2n-k-1}  \bar z^{k+2}}\bigg],
\end{align*}
where
\begin{align*}
&{\partial \over \partial x_{\alpha+1}} \bigg[ {\bar\sigma\over z^{2n-k}   \bar z^{k+2}} \bigg]\\
&={1\over z^{2n-k}   \bar z^{k+2}}{\partial \bar \sigma \over \partial x_{\alpha+1}}-{(2n-k)\bar\sigma\over z^{2n-k+1}   \bar z^{k+2}}{\partial z\over \partial x_{\alpha+1}}-{(k+2)\bar\sigma\over z^{2n-k}   \bar z^{k+3}}{\partial \bar z\over \partial x_{\alpha+1}}\\
&=-{1\over z^{2n-k}   \bar z^{k+2}} {\bf i}_\alpha-{(2n-k)\bar\sigma\over z^{2n-k+1}   \bar z^{k+2}}{x_{\alpha+1}\over |\operatorname{Im}\sigma|} {\bf i}_\alpha+{(k+2)\bar\sigma\over z^{2n-k}   \bar z^{k+3}}{x_{\alpha+1}\over |\operatorname{Im}\sigma|} {\bf i}_\alpha,
\end{align*}
and
\begin{align*}
{\partial \over \partial x_{\alpha+1}} \bigg[ {1\over z^{2n-k-1}   \bar z^{k+2}} \bigg]
&=-{(2n-k-1)\over z^{2n-k}   \bar z^{k+2}}{\partial z\over \partial x_{\alpha+1}}-{k+2\over z^{2n-k-1}   \bar z^{k+3}}{\partial \bar z\over \partial x_{\alpha+1}}\\
&=-{(2n-k-1)\over z^{2n-k}   \bar z^{k+2}}{x_{\alpha+1}\over |\operatorname{Im}\sigma|}{\bf i}_\alpha+{k+2\over z^{2n-k-1}   \bar z^{k+3}}{x_{\alpha+1}\over |\operatorname{Im}\sigma|}{\bf i}_\alpha.
\end{align*}
Therefore,
\begin{align*}
 &Y_{4l+j} K(g)\\
 & =2y_{4l+j} {\partial s\over \partial x_1}( |y|^2, t_1,t_2 , t_3 )  +2 \sum_{\alpha=1}^3\sum_{k=1}^{4} b_{kj}^\alpha y_{4l+k} {\partial s\over \partial x_{\alpha+1}}( |y|^2, t_1,t_2 , t_3 )\\
 &=2y_{4l+j}\Bigg[ -c_{n-1}\sum_{k=0}^{2n-1}     (2n-1)! (2n-k)(k+1)   {(|y|^2-t)\over (|y|^2+|t|{\bf i})^{2n-k+1}(|y|^2-|t|{\bf i})^{k+2}}
\\
 &\qquad\qquad\quad+c_{n-1}\sum_{k=0}^{2n-2}     (2n-1)! (2n-k-1)(k+1)   {1\over (|y|^2+|t|{\bf i})^{2n-k}(|y|^2-|t|{\bf i})^{k+2}} \Bigg]\\
 &\quad+2 \sum_{\alpha=1}^3\sum_{k=1}^{4} b_{kj}^\alpha y_{4l+k}\Bigg\{c_{n-1}\sum_{k=0}^{2n-2}   (2n-2)!  (2n-k-1) (k+1)\\
 &\qquad\qquad  \bigg[ -{1\over (|y|^2+|t|{\bf i})^{2n-k}   (|y|^2-|t|{\bf i})^{k+2}} {\bf i}_\alpha
 -{(2n-k)(|y|^2-t)\over (|y|^2+|t|{\bf i})^{2n-k+1}   (|y|^2-|t|{\bf i})^{k+2}}{t_\alpha\over|t|} {\bf i}_\alpha\\
 &\qquad\qquad\qquad +{(k+2)(|y|^2-t)\over (|y|^2+|t|{\bf i})^{2n-k}   (|y|^2-|t|{\bf i})^{k+3}}{t_\alpha\over|t|} {\bf i}_\alpha \bigg]\\
 &\qquad\quad -  c_{n-1}\sum_{k=0}^{2n-3}   (2n-2)!  (2n-k-2) (k+1) \bigg[-{(2n-k-1)\over (|y|^2+|t|{\bf i})^{2n-k}   (|y|^2-|t|{\bf i})^{k+2}}{t_\alpha\over|t|}{\bf i}_\alpha \\
 &\qquad\qquad\qquad\qquad\qquad
 +{k+2\over (|y|^2+|t|{\bf i})^{2n-k-1}   (|y|^2-|t|{\bf i})^{k+3}}{t_\alpha\over|t|}{\bf i}_\alpha \bigg]  \Bigg\}\\
 &=:I+II.
\end{align*}
For the term $I$, by definition, we have
\begin{align*}
|I|&\leq 2c_{n-1}\sum_{k=0}^{2n-1}     (2n-1)! (2n-k)(k+1)   {(|y|^4+|t|^2)^{1\over 2}|y_{4l+j}|\over (|y|^4+|t|^2)^{2n-k+1\over 2}(|y|^4+|t|^2)^{k+2\over 2}}  \\
 &\quad+(2n-1)c_{n-1}\sum_{k=0}^{2n-2}     (2n-2)! (2n-k-1)(k+1)   {|y_{4l+j}|\over (|y|^4+|t|^2)^{2n-k\over 2}(|y|^4+|t|^2)^{k+2\over 2}}\\
 &\leq C{1\over (|y|^4+|t|^2)^{n+{3\over 4}}}\\
 &=C{1\over \|g\|^{Q+1}}.
\end{align*}

For the term $II$, we find
\begin{align*}
|II|&\leq 2 \sum_{\alpha=1}^3\sum_{k=1}^{4n-4} \big|b_{kj}^\alpha\big| |y_{4l+k}|
\\
&\quad\times\Bigg\{c_{n-1}\sum_{k=0}^{2n-2}   (2n-2)!  (2n-k-1) (k+1) \bigg[ {1\over (|y|^4+|t|^2)^{2n-k\over 2}   (|y|^4+|t|^2)^{k+2\over 2}}\\
 &
 \qquad\qquad+{(2n-k)(|y|^4+|t|^2)^{1\over 2}\over (|y|^4+|t|^2)^{2n-k+1\over 2}   (|y|^4+|t|^2)^{k+2\over 2}}{|t_\alpha|\over|t|} +{(k+2)(|y|^4+|t|^2)^{1\over 2}\over (|y|^4+|t|^2)^{2n-k\over 2}   (|y|^4+|t|^2)^{k+3\over 2}}{|t_\alpha|\over|t|}  \bigg] \\
 &\qquad+ c_{n-1}\sum_{k=0}^{2n-3}   (2n-2)!  (2n-k-2) (k+1) \bigg[{(2n-k-1)\over (|y|^4+|t|^2)^{2n-k\over 2}   (|y|^4+|t|^2)^{k+2\over 2}}{|t_\alpha|\over|t|}  \\
 &\qquad\qquad\qquad\qquad
+{k+2\over (|y|^4+|t|^2)^{2n-k-1\over 2}   (|y|^4+|t|^2)^{k+3\over 2}}{|t_\alpha|\over|t|} \bigg]  \Bigg\}\\
 &\leq 2 \sum_{\alpha=1}^3\sum_{k=1}^{4n-4} \big|b_{kj}^\alpha\big| \Bigg\{c_{n-1}\sum_{k=0}^{2n-2}   (2n-2)!  (2n-k-1) (k+1) \\
 &\qquad\qquad\qquad\qquad\qquad\qquad\times  \bigg[ {1\over (|y|^4+|t|^2)^{n+{3\over 4}} }
 +{(2n-k)\over (|y|^4+|t|^2)^{n+{3\over 4}}}
 +{(k+2)\over (|y|^4+|t|^2)^{n+{3\over 4}}}  \bigg] \\
 &\qquad
 + c_{n-1}\sum_{k=0}^{2n-3}   (2n-2)!  (2n-k-2) (k+1)
\times\bigg[{(2n-k-1)\over (|y|^4+|t|^2)^{n+{3\over 4}}} +{(2n-k)\over (|y|^4+|t|^2)^{n+{3\over 4}}} \bigg]  \Bigg\}\\
 &\leq C{1\over (|y|^4+|t|^2)^{n+{3\over 4}}}\\
 &=C{1\over \|g\|^{Q+1}},
\end{align*}
which implies \eqref{gradient} holds.

Therefore, combining the estimates of $I$ and $II$ we obtain that
\begin{align*}
\big| Y_{4l+j} K(g)\big|\leq C{1\over \|g\|^{Q+1}}.
\end{align*}
Thus
\begin{align*}
\big| Y_{4l+j} K(x\cdot u)\big|\leq C{1\over \|x\cdot u\|^{Q+1}}
=C{1\over \rho(u, x^{-1})^{Q+1}}.
\end{align*}

Since $\rho$ is a quasi-distance, by \eqref{rho}, \eqref{norm-equ} and the fact that $\rho(g_0,h)\geq c\rho(g,g_0)$, we have
\begin{align*}
\rho(u, x^{-1})&\geq {1\over C_\rho}\rho(0,x^{-1})-\rho(0,u)
\geq{1\over C_\rho}\rho(g_0,h)-{\tilde C_d\over C_d}\rho(g,g_0)\\
&\geq{1\over C_\rho}\rho(g_0,h)-{\tilde C_d\over cC_d}\rho(g_0,h)\\
&=\Big({1\over C_\rho}-{\tilde C_d\over cC_d}\Big)\rho(g_0,h).
\end{align*}
Consequently,
\begin{align*}
\big| Y_{4l+j} K(x\cdot u)\big|\lesssim {1\over \rho(g_0,h)^{Q+1}}.
\end{align*}
Together with \eqref{kk} we can see
\begin{equation*}
|K(g, h)-K(g_0,h)|\lesssim  {\rho(  g,g_0 )\over \rho(g_0,h)^{Q+1}}.
\end{equation*}

Similarly, if $\rho(g,h_0)\geq c \rho(h,h_0)$, we can obtain
$$ |K(g,h) - K(g,h_0) |\lesssim   { \rho(h, h_0)\over \rho(g,h_0)^{Q+1} }.$$
The proof of Theorem \ref{thm-ksize} is complete.
\end{proof}

\begin{proof}[Proof of Theorem \ref{thm-lowerbd}]


We prove this theorem by 4 steps, using the idea in \cite{DLLW}.

\medskip
{\bf Step 1.} We claim that: ``there exists a point $g_0$ on the unit sphere of $\mathscr H^{n-1}$ such that $K(g_0)\not=0$''.

\medskip
To see this, we choose $g_0 = \Big( {{\bf i} \over  \sqrt2}, {y\over \sqrt[4]{2}} \Big) \in \mathscr H^{n-1}$ such that $|y|=1$. It is clear that this $g_0$ is on the unit sphere of $\mathscr H^{n-1}$.

By \eqref{kg},  we have
$$K(\delta_{\sqrt[4] 2}~g_0)=s(1+{\bf i}).$$

Based on the equality \eqref{s sigma1} in the proof of Theorem \ref{thm-ksize}, we see that for $\sigma=1+{\bf i} $, we have
$z=1+{\bf i}$ such that
\begin{align*}
s(1+{\bf i})
&= c_{n-1} {2   (2n-2)! \over |z|^4    (z-\bar z)^3 }\\
&\qquad\times\Bigg\{ \bigg[ {\bar z^2\over z^{2n-2}}\big( z+(2n-1){z-\bar z\over2} \big) - { z^2\over\bar z^{2n-2}}\big( \bar z-(2n-1){z-\bar z\over2} \big) \bigg]   {\overline \sigma} \nonumber\\
&\qquad\qquad\qquad-
  \bigg[ {\bar z^2\over z^{2n-3}}\big( z+(2n-2){z-\bar z\over2} \big) - { z^2\over\bar z^{2n-3}}\big( \bar z-(2n-2){z-\bar z\over2} \big)\bigg] \Bigg\}\\
  &:= -c_{n-1} {  (2n-2)! \over 16{\bf i} } \times A.  \nonumber
\end{align*}

For the term $A$, we further have
%
%
\begin{align*}
A&=  \bigg[ {(1-{\bf i})^2\over (1+{\bf i})^{2n-2}}\big( 1+{\bf i}+(2n-1) {\bf i} \big) - { (1+{\bf i})^2\over(1-{\bf i})^{2n-2}}\big( 1-{\bf i}-(2n-1) {\bf i} \big) \bigg]   (1-{\bf i}) \nonumber\\
&\qquad-
  \bigg[ {(1-{\bf i})^2\over (1+{\bf i})^{2n-3}}\big( 1+{\bf i}+(2n-2) {\bf i} \big) - { (1+{\bf i})^2\over (1-{\bf i})^{2n-3}}\big( (1-{\bf i})-(2n-2) {\bf i} \big)\bigg] \\
  &= { (1-{\bf i})^3\over (1+{\bf i})^{2n-2}  } (1+2n{\bf i})-
  { (1+{\bf i})^2\over (1-{\bf i})^{2n-3}  } (1-2n{\bf i})
  - { (1-{\bf i})^2\over (1+{\bf i})^{2n-3}  } (1+(2n-1){\bf i})\\
  &\qquad +  { (1+{\bf i})^2\over (1-{\bf i})^{2n-3}  } (1-(2n-1){\bf i})\\
  &= { (1-{\bf i})^3(1+ {\bf i})^2\over (1+{\bf i})^{2n}  } (1+2n{\bf i})-
   { (1-{\bf i})^2(1+{\bf i})^3\over (1+{\bf i})^{2n}  } (1+(2n-1){\bf i})\\
&\qquad+  { (1+{\bf i})^2 (1-{\bf i})^3\over (1-{\bf i})^{2n}  }[ 1- (2n-1){\bf i} -(1-2n{\bf i}) ]\\
&={4\over  (1+{\bf i})^{2n}}\big[(1-{\bf i}) (1+2n{\bf i})-(1+{\bf i}) (1+(2n-1){\bf i})\big]
+ {4\over  (1-{\bf i})^{2n}} (1-{\bf i}) {\bf i}\\
&= {4\over  (1+{\bf i})^{2n}}(4n-1-{\bf i})+ {4\over  (1-{\bf i})^{2n}}({\bf i}+1)\\
&={4\over (2i)^n}\big[4n-1-{\bf i}+(-1)^n({\bf i}+1)\big]\not=0.
\end{align*}
Therefore,
$$K (\delta_{\sqrt[4] 2}~g_0) \not=0.$$
 Then by Corollary \ref {cor homo}, we see that $K( g_0)\not=0$, which shows that the claim holds.

\medskip
{\bf Step 2.} We claim that: ``there is a positive constant $C$ such that for every $R>0$ and every $g\in \mathscr H^{n-1}$, there exists a point $g_*$ in $\mathscr H^{n-1}$ such that $\rho(g,g_*)=R$ and that  $|K(g,g_*)|\approx {1\over R^Q}$ ''.

To see this, for every $R>0$ and every $g\in \mathscr H^{n-1}$, we set
$$ g_* = g \cdot \delta_R (g_0^{-1}). $$

Then it is clear that
$$ \rho(g,g_*) = \rho(g, g \cdot \delta_R (g_0^{-1})) = R \rho({\bf 0},g_0^{-1}) =  R \rho({\bf 0},g_0)=R.  $$

Next, we note that  from Corollary \ref{cor homo}
\begin{align*}
K(g,g_*) = K(  g, g \cdot \delta_R (g_0^{-1}) ) = K({\bf 0}, \delta_R (g_0^{-1})) = R^{-Q}  K({\bf 0},  g_0^{-1}) = R^{-Q} K(g_0).
\end{align*}
Hence, we obtain that
$$ |K(g,g_*)|=  R^{-Q} |K(g_0)|=: C_0 R^{-Q} $$
and since from {\bf Step 1} we know that $K(g_0)$ is non-zero.

\medskip
{\bf Step 3.} We claim that: there exist positive constants $C$ and $R$ such that for every ball $B=B(g,r)\subset \mathscr H^{n-1}$ with $r>0$ and $g\in \mathscr H^{n-1}$, there exists another ball $\tilde B= B(g_*,r)$ in $\mathscr H^{n-1}$ such that $\rho(g,g_*)=Rr$ and that  for every
$g_1\in B$ and $g_2\in \tilde B$,
$$|K(g_1,g_2)|\geq {C\over \rho(g_1,g_2)^Q}.$$

To see this, we first note that $K$ satisfies H\"older's regularity. Hence, for the point $g_0$ obtained  in {\bf Step 1},  there exists a small positive constant $\epsilon_0<1$ such that
\begin{align}\label{g0n}
 K(\tilde g) \not=0\quad{\rm and}\quad |K(\tilde g)|>{1\over 2}|K(g_0)|
 \end{align}
for all $\tilde g\in  B(g_0,2C_\rho\epsilon_0)$.

Now for every $B=B(g,r)\subset \mathscr H^{n-1}$ with $r>0$ and $g\in \mathscr H^{n-1}$, by using the argument in {\bf Step 2} we see that there exists a point $g_*=g \cdot \delta_{Rr} (g_0^{-1})$ in $\mathscr H^{n-1}$ such that $\rho(g,g_*)=Rr$ and that  $|K(g,g_*)|\approx {1\over (Rr)^Q}$, where $R$ is chosen to be large enough such that
$ {1\over 2\epsilon_0}<R<{1\over\epsilon_0} $.

Then for every $g_1\in B$ there exists $g'_1\in B({\bf 0},\epsilon_0)$ such that we can write $g_1= g\cdot \delta_{Rr} (g'_1)$. Also, for every $g_2\in B(g_*, r)$, there exists $g'_2\in B({\bf 0},\epsilon_0)$ such that we can write $g_2= g_*\cdot \delta_{Rr} (g'_2)$.
Now we have
\begin{align*}
K(g_1,g_2) &= K\big( g\cdot \delta_{Rr} (g'_1), g_*\cdot \delta_{Rr} (g'_2) \big)
=K\big( g\cdot \delta_{Rr} (g'_1), g \cdot \delta_{Rr} (g_0^{-1})\cdot \delta_{Rr} (g'_2) \big)\\
&=K\big(  \delta_{Rr} (g'_1),  \delta_{Rr} (g_0^{-1})\cdot \delta_{Rr} (g'_2) \big)=K\big(  \delta_{Rr} (g'_1),  \delta_{Rr} (g_0^{-1}\cdot  g'_2) \big)\\
&= (Rr)^{-Q}K\big(   g'_1,  g_0^{-1}\cdot  g'_2 \big)\\
&= (Rr)^{-Q}K\big(    (g'_2)^{-1}\cdot g_0 \cdot g'_1   \big).
\end{align*}
Next, note that
\begin{align*}
\rho\big(    (g'_2)^{-1}\cdot g_0 \cdot g'_1, g_0   \big) &= \rho\big(     g_0 \cdot g'_1, g'_2\cdot g_0   \big)\leq C_\rho\big[ \rho\big(     g_0 \cdot g'_1,  g_0   \big) + \rho\big(     g_0, g'_2\cdot g_0   \big) \big]\\
&\leq C_\rho\big[ \rho\big(      g'_1,  0   \big) + \rho\big(     0, g'_2  \big) \big]\leq 2C_\rho\epsilon_0,
\end{align*}
which shows that $(g'_2)^{-1}\cdot g_0 \cdot g'_1\in B(g_0,2C_\rho\epsilon_0)$.
Then by \eqref{g0n}, we have
$$|K\big(    (g'_2)^{-1}\cdot g_0 \cdot g'_1   \big)|>{1\over 2}|K(g_0)|.$$
Therefore,
$$K(g_1,g_2)>{1\over 2(Rr)^Q}|K(g_0)|\geq C(K,C_\rho){1\over \rho(g_1, g_2)^Q}, $$
for any $g_1\in B(g,r)$ and any $g_2\in B(g_*, r)$.

\medskip
{\bf Step 4.}
Since $K$ is continuous on $\mathscr H^{n-1}\setminus\{{\bf 0}\}$, by {\bf Step 1.} we can see that there exists a compact set $\Omega$ on the unit sphere $\mathcal S({\bf 0},1)$ of $\mathscr H^{n-1}$ with $m(\Omega)>0$, where $m$ is the Radon measure on $\mathcal S({\bf 0},1)$, such that
$K(g)\not=0$ for any $g\in\Omega$. Then by the similar proof to that of \cite[Theorem 1.1]{CDLW}, we can
find a positive constant $r_0$ and a ``twisted truncated sector'' $S_g\subset \mathscr H^{n-1}$
such that
$$ \inf_{g'\in S_g} \rho(g,g')=r_0 $$ and that for every $g_1\in B(g,1)$ and $g_2\in S_g$ we have
\begin{align*}
|K(g_1, g_2)|\geq  {C(K, C_\rho)\over\rho(g_1,g_2)^{Q}}.
\end{align*}
Moreover, this sector $S_g$ is regular in the sense that $|S_g|=\infty$ and that for every
$R_2>R_1>2r_0$
$$ \big| \big(B(g,R_2)\backslash B(g,R_1)\big)  \cap S_g\big | \approx  \big| B(g,R_2)\backslash B(g,R_1)\big|$$
with the implicit constants  independent of $g$ and $R_1,R_2$.
The proof of Theorem \ref {thm-lowerbd} is complete.
\end{proof}
\begin{remark}
From the proof of Theorem \ref {thm-lowerbd}, we see that we have the following version of kernel
lower bound concerning the real coordinate of each component.

Since $K(g,h)$ is $\mathbb H$-valued, we write
$$ K(g,h) = K_1(g,h)+K_2(g,h){\bf i}+K_3(g,h){\bf j}+K_4(g,h){\bf k}, $$
where each $K_i(g,h)$ is real-valued, $i=1,2,3,4$.

Then following kernel lower bound in  Theorem \ref {thm-lowerbd} and using the idea of Hyt\"onen \cite{Hy} we see that at least one of the $K_i$ above satisfies the following argument:

There exist positive constants $3\le A_1\le A_2$ such that for any ball $B:=B(g_0, r)\subset \mathscr H^{n-1}$, there exist another ball $\widetilde B:=B(h_0, r)\subset \mathscr H^{n-1}$
such that $A_1 r\le d(g_0, h_0)\le A_2 r$, 
and  for all $(g,h)\in ( B\times \widetilde{B})$,
$K(g, h)$ does not change sign
and
\begin{equation}\label{e-assump cz ker low bdd}
|K(g, h)|\geq  {C\over\rho(g,h)^{Q}}.
\end{equation}

\end{remark}

\begin{proof}[Proof of Theorem \ref{thm-commutator}]
The argument (i) follows from the kernel lower bound obtained in Theorem \ref{thm-lowerbd} and
the standard proof in \cite{Hy} (also can be achieved by following \cite{DGKLWY}).

It suffices to prove the ``only if'' part of (ii), since the ``if'' part follows from the argument in \cite{CDLW}.

To see this, we seek to get a contradiction by using the idea (due to M. Lacey) stated at the beginning of
 proof of (2) of Theorem 1.1 in \cite{DLLWW}, that is, in its simplest form, there is no bounded operator $T \;:\; \ell^{p} (\mathbb N) \to \ell^{p} (\mathbb N)$ with $Te_{j } = T e_{k} \neq 0$ for all $j,k\in \mathbb N$.  Here, $e_{j}$ is the
standard basis for $\ell^{p} (\mathbb N)$.

To get to this contradiction, we just need to use the kernel lower bound in  \eqref{e-assump cz ker low bdd}, and then combine the argument used in the proof of (2) of Theorem 1.1 in \cite{DLLWW}.
Repeating the process there almost step by step, we obtain the ``only if'' part.
We leave the details to readers.
\end{proof}

\bigskip
{\bf Acknowledgement:} Chang is supported by NSF grant DMS-1408839 and a McDevitt Endowment Fund at Georgetown University. Duong and Li are supported by the Australian Research Council (ARC) through the research grants DP170101060 and DP190100970 and by Macquarie University Research Seeding Grant.
 Wang is supported by National Nature Science Foundation in China (No. 11571305).
 Wu is supported by NSF of China (Grants No. 11671185 and No. 11701250), the Natural Science Foundation of Shandong Province (No. ZR2018LA002 and No. ZR2019YQ04).

\bibliographystyle{amsplain}

\end{document}